\documentclass{NotesRM}

\usepackage{enumitem}
\usepackage{tikz}
\usetikzlibrary{arrows,calc}
\newcommand{\thp}{{\mathbf T}}
\newcommand{\cnd}{{\mathbf Q}}
\newcommand{\pp}{{\mathsf P}}
\newcommand{\spm}{{\mathcal S}}
\newcommand{\thm}{{\mathcal T}}
\newcommand{\lgcp}{\text{lGCP}}
\newcommand{\vx}{{\mathbf x}}
\newcommand{\vy}{{\mathbf y}}
\allowdisplaybreaks
\begin{document}

%%%%%%%%%%%%%%%%%
% AUTHOR and TITLE
%%%%%%%%%%%%%%%%%
\author{Mathias Rafler}
\title{General thinning characterizations of distributions and point processes}
%%%%%%%%%%%%%%%%%

%%%%%%%%%%%%%%%%%
% HEAD MATTER
%%%%%%%%%%%%%%%%%

\maketitle

%%%%%%%%%%%%%%%%%
% ABSTRACT
%%%%%%%%%%%%%%%%%

\begin{abstract}
For general thinning procedures, its inverse operation, the condensing, is studied and a link to integration-by-parts formulas is established. This extends the recent results on that link for independent thinnings of point processes to general thinnings of finite point processes. In particular, the classical integration-by-parts formulas appear as the example of independent thinnings. Moreover, the representation of the splitting kernel of finite point processes in terms of its reduced Palm kernels is extended to general thinnings.

This link is studied in the context of discrete random variables and yields analogue characterizations of their distributions. Results on independent thinnings are complemented by a discrete stick breaking characterization of distributions.

\end{abstract}

%%%%%%%
% STRUCTURE
%%%%%%%
% Inroduction and general questions
%% Objects and mathematical description
% Thinnings of random variables
%% Thinned laws
%% Thinning characterization
% Thinnings of point processes
%% Setup
%% Thinned point processes
%% Thinnings characterization

%%%%%%%%%%%%%%%%%
% MAIN MATTER
%%%%%%%%%%%%%%%%%

%%%%%%%%%%%%%%%%%
\section{Introduction}
%%%%%%%%%%%%%%%%%

Given a random variable $N$ on the set $\N_0$ of non-negative integers representing the number of certain objects, let $N_\ast$ be the number of objects of a subsample of the $N$ objects and $N^\ast$ its remainder such that $N_\ast+N^\ast=N$. Suppose that only $N_\ast$ is observed, then the natural question concerns inference from $N_\ast$ to $N$ or, equivalently, to $N^\ast$. The former is referred to as condensing, the latter as splitting. Typically, $N_\ast$ and the remainder $N^\ast$ are not independent random variables.

Usually simple exercises are the computation of e.g. the joint distribution of $N_\ast$ and its remainder $N^\ast$, such that the splitting equation is
\begin{align}	\label{eq:intro:split}
	\Ex g(N_\ast,N^\ast) = \Ex g(N_\ast,N-N_\ast).
\end{align}
The expectation on the right hand side requires in fact the knowledge of the distribution of $N$ and the sampling rule; whereas the expectation on the left hand side requires the law of $N_\ast$ and the conditional law of $N^\ast$ given $N_\ast$ with the law of $N$ being only implicit in the law of $N_\ast$.

Such a situation was discussed in the context of point processes with an independent sampling mechanism, see e.g~\cite{FF00,bN13a,N,NRZ15}: Suppose that $N$ is a point process realizing a possibly infinite number of points in some Polish space. Colour these points independently of each other with probability $q$ red and with probability $1-q$ blue. The point configuration $N_q$ of red points is called (independent) $q$-thinning of $N$, and the joint law together with the configuration of blue points $N_q^\ast$ is the (independent) $q$-splitting. In that context, of interest is the conditional law of $N_q^\ast$ given $N_q$, which is called $q$-splitting kernel. For the Poisson process, it is well-known that $N_q$ and $N_q^\ast$ are Poisson processes both and independent of each other. Moreover, the independence characterizes the Poisson process~\cite{kF75}. Ambartzumian touched this topic for Gibbs processes and asked under which conditions a Gibbs process may appear as a thinning of another process~\cite{rA91}. Recently, splittings were studied in the larger class of Papangelou processes (see e.g.~\cite{hZ09}), which are given by an integration-by-parts formula. In~\cite{NRZ15}, it was shown that a splitting equation of type~\eqref{eq:intro:split} with the independent splitting is equivalent to an integration-by-parts formula, which is known to be equivalent to the Dobrushin-Lanford-Ruelle equations~\cite{NZ79}. Thus, the static splitting mechanism is related to a reversibility condition for a spatial birth-and-death process.

These results, apart from the link to the Dobrushin-Lanford-Ruelle equations, simplified to random variables, yield characterizations of Poisson, binomial and negative binomial distributions with explicit representations of all involved objects. However, the independent sampling mechanism of colouring each object red and blue independent of each other, is just one choice of choosing a subsample. In Section~\ref{sect:trv} the aim is firstly to show the link between sampling and condensing and integration-by-parts, and secondly to characterize pairs of sampling mechanism and condensing mechanisms which characterize a law of a random variable $N$. Besides the independent sampling, this allows to deal with a discrete stick breaking as an example.

These discussions of the discrete setting carry over to multivariate situations, i.e. thinning of a random element on a hypercube or a random configuration of the Ising model on a finite graph. Both are discrete partially ordered sets, and thinning means to choose at random a configuration which is smaller in this partial order. This study is subject to future work. The setup in Section~\ref{sect:fpp} is more general, when finite point processes on Polish spaces are considered. Firstly, the splitting kernel of a finite point process is essentially given by its reduced Palm distributions, which generalizes a relation known for independent splitiings. Secondly, an integration-by-parts formula suited to the given thinning is derived, which contains the known formulas as a special case for independent $q$-thinnings. Finally, the question of obtaining a point process from a thinning and a splitting is addressed.

%%%%%%%%%%%%%%%%%
\section{Thinnings of distributions on $\N_0$}	\label{sect:trv}
%%%%%%%%%%%%%%%%%
\subsection{Thinned and condensed laws}	
%%%%%%%%%%%%%
 
\begin{definition}
A \emph{thinning matrix} $\thp$ is a stochastic matrix on $\N_0$ such that $\thp_{n,k}=0$ for all indices $k>n\geq 0$. $\thp$ is said to be positive if $\thp_{n,k}>0$ for all $k\leq n$ and is said to connect to all lower levels if for each $n\geq 0$ and $k\leq n$ there exist $n=j_0,j_1,\ldots,j_l=k$ such that $\thp_{j_0,j_1},\ldots,\thp_{j_{l-1},j_k}>0$.
\end{definition}
Given a random variable $N$ with distribution $\nu$, $\thp_{n,\scdot}$ is the conditional law of $N_\ast$ given $N=n$. $\thp$ is a lower triangular matrix. Subsequently, let $\nu'$ denote the law of $N_\ast$, which is given by $\nu'=\nu\thp$. Basic examples of thinning matrices follow.

\begin{example}
\begin{enumerate}
	\item The matrix related to the \emph{independent $q$-thinning} for some $q\in(0,1)$ is
		\begin{align*}
			\thp_{n,k} = \begin{cases}
						\binom{n}{k} q^k(1-q)^{n-k}, & k=0,1,\ldots,n;\\
						0 & k>n.
					\end{cases}
		\end{align*}
		The independent $q$-thinning is the mechanism described in the introduction for dividing random point configurations into the red and blue subconfiguration by independent coin tosses.
	\item The \emph{uniform thinning} is given by
		\begin{align*}
			\thp_{n,k} = \begin{cases}
						\frac{1}{n+1}, & k=0,1,\ldots,n;\\
						0 & k>n.
					\end{cases}
		\end{align*}
		This mechanism is simply understood as follows: If the $n$ objects were arranged in a line, then a barrier chosen uniformly among the $n-1$ gaps and the two boundaries produces two subsamples to the left and to the right of the barrier. This is a version of a \emph{discrete stick-breaking}.
	\item Further examples are taking almost nothing, or taking all or nothing
		\begin{align*}
			\thp_{n,k} = \begin{cases}
				1, & k=n=0;\\	q, & k=n\geq 1;\\	1-q, & k=n-1\geq 0; \\	0, & \text{else;}
			\end{cases}	\qquad
			\thp_{n,k} = \begin{cases}
				1, & k=n=0;\\	q, & k=n\geq 1;\\	1-q, & k=0, n\geq 1; \\	0, & \text{else.}
			\end{cases}
		\end{align*}
\end{enumerate}
Note that in case of $q=\nicefrac{1}{2}$, the first two examples are certain ``uniform'' sampling mechanisms. The last two examples are not primer examples for having too many zeros. While the first of the two matrices connects to all lower levels, the second does not.
\end{example}

The right hand side of Equation~\eqref{eq:intro:split} is given explicitly by
\begin{align*}
	\Ex g(N_\ast,N-N_\ast) = \sum_{n,k\in\N_0} \nu_n \thp_{n,k} g(k,n-k)
\end{align*}
for any non-negative function $g$ on $\N_0\times \N_0$. Its disintegration with respect to $N_\ast$ yields another matrix $\Upsilon$,
\begin{align}	\label{eq:th-sp:disint}
	\Ex g(N_\ast,N-N_\ast) = \Ex g(N_\ast,N^\ast) = \sum_{k,l\in\N_0} \nu'_k \Upsilon_{k,l} g(k,l),
\end{align}
where $\nu'\defeq\nu \thp$ is the distribution on $N_\ast$. Note that Equation~\eqref{eq:th-sp:disint} is equivalent to
\begin{align}	\label{eq:th-sp:depconv}
	\Ex g(N) = \Ex g(N_\ast+N^\ast) = \sum_{k,l\in\N_0} \nu'_k \Upsilon_{k,l} g(k+l).
\end{align}

\begin{definition}
Equation~\eqref{eq:th-sp:disint} is called \emph{splitting equation}, and the stochastic matrix $\Upsilon$ therein \emph{splitting matrix}. Equation~\eqref{eq:th-sp:depconv} is called \emph{dependent convolution equation}. Moreover, call $\cnd_{k,n}\defeq \Upsilon_{k,n-k}$, $k\leq n$, and $\cnd_{k,n}\defeq 0$, $k>n$, condensation matrix.
\end{definition}

Thus, the matrix $\cnd$ is the conditional law of $N$ given $N_\ast$, whereas $\Upsilon$ is the conditional law of $N^\ast$ given $N_\ast$. In particular, Equation~\eqref{eq:th-sp:depconv} reads in terms of $\cnd$ as
\begin{align*}
	\Ex g(N) = \sum_{k,n\in\N_0} \nu'_k \cnd_{k,n} g(n).
\end{align*}
$\cnd$ is easily calculated by means of Bayes' rule. In terms of Bayesian statistics, $\nu$ is the prior distribution of $N$ and, given the partial observation $N_\ast=k$, $\cnd_{k,\scdot}$ is the posterior distribution of $N$.

\begin{remark}
Let $\thp$ be the independent $q$-thinning matrix for some $q\in(0,1)$. The results of~\cite{NRZ15} show that if $\nu$ is a Poisson, binomial or negative binomial distribution, then also the rows of $\Upsilon$ and therefore also those of $\cnd$ are of the same distribution apart from a shift in case of $\cnd$. Moreover, if $\thp$ and $\cnd$ are fixed that way, they determine a unique distribution and the related splitting equations are equivalent to integration-by-parts formulas for the distributions. In Subsection~\ref{sect:gp:ex} these relations are summarized in more detail.
\end{remark}

%In the following, assume that the thinning matrix is always non-trivial in the sense that $\thp_{0,0}$ is the only component of $\thp$ with entry 1.

\begin{proposition}
Let $\thp$ be a thinning matrix and $\nu$ be any distribution on $\N_0$. Then the condensation matrix $\cnd$ is related to $\thp$ and $\nu$ via the balance equations
\begin{align}	\label{eq:gp:balance}
	\nu'_k \cnd_{k,n} = \nu_n \thp_{n,k}
\end{align}
for all $k,n\in\N_0$, where $\nu'=\nu \thp$. Particularly, $\cnd_{k,n}=0$ whenever $k>n$.
\end{proposition}

\begin{proof}
Choosing $g=1_{\{(k,n-k)\}}$ in the splitting equation~\eqref{eq:th-sp:disint} yields
\begin{equation*}
	\nu_n \thp_{n,k} = \Ex[1_{\{(k,n-k)\}}(N_\ast,N-N_\ast)] 
		= \Ex[1_{\{(k,n-k)\}}(N_\ast,N^\ast)]
		= \nu_k' \Upsilon_{k,n-k} 
		= \nu_k' \cnd_{k,n}.\qedhere
\end{equation*}
\end{proof}

Thus, $\cnd$ balances the mass of $\nu$ transported by $\thp$. Relation~\eqref{eq:gp:balance} is similar to the balancing in the time inversion of Markov chains apart from the fact that this is not considered in an equilibrium regime here.

The condensation matrix is easily calculated via
\begin{align*}
	\cnd_{k,n} = \nu_n \cdot \thp_{n,k} \cdot \frac{1}{\nu_k'}
\end{align*}
whenever $\nu_k'$ is positive. Otherwise the corresponding row is not of interest, since no matter what the size of the full sample was, there is no chance that a subsample has size $k$. In such pathological cases choose e.g. $Q_{k,k}=1$. Particularly, the positivity of $\thp$ implies positivity of $\cnd$ up to some $n_0\in\N\cup\{\infty\}$.

%If $\thp$ connects to all lower levels and the support of $\nu$ is connected, then also the support of $\nu'$ is connected.

In the remainder of this subsection, an integration-by-parts formula is shown to hold for $\nu$ as well as two alternating cycle conditions for $\thp$ and $\cnd$ in the spirit of Kolmogorov's cycle condition for Markov chains.

Clearly, certain connectivity conditions on the support of $\nu$ in terms of $\thp$ are needed. To simplify this part of the discussion, $\nu$ will be assumed to have a support without holes (but may be finite). Note that in this case the conditions on $\thp$ in the following theorem imply that $\cnd_{n,n}, \cnd_{n,n+1}>0$ for all $n$ smaller than the smallest index $n_0$ such that $\nu_{n_0}=0$ reduced by one.

\begin{theorem}[Integration-by-parts formula]	\label{thm:dn:ibpf}
Let $\thp$ be a thinning matrix such that for all $n\in\N$, $\thp_{n,n},\thp_{n,n-1}>0$, $N$ be a random variable with law $\nu$ such that if for some $n_0$, $\nu_{n_0}=0$ holds, then $\nu_n=0$ for all $n\geq n_0$, and $\cnd$ be the corresponding condensation matrix. Then for all non-negative functions $g$ the following integration-by-parts formula holds
\begin{align*}
	\Ex\left[ g(N) \frac{\thp_{N,N-1}}{\thp_{N-1,N-1}} \right] 
		= \Ex\left[ g(N+1) \frac{\cnd_{N,N+1}}{\cnd_{N,N}} \right].
\end{align*}
\end{theorem}
$\pi(n)\defeq\nicefrac{\cnd_{n,n+1}}{\cnd_{n,n}}$ is called Papangelou function.

\begin{proof}
Let $g:\N_0\to\R_+$ be bounded, then by a double application of~\eqref{eq:gp:balance} and an index shift,
\begin{align*}
	\sum_{n\geq 1} g(n) \frac{\thp_{n,n-1}}{\thp_{n-1,n-1}} \nu_n 
		&= \sum_{n\geq 1} g(n) \frac{\cnd_{n-1,n}}{\thp_{n-1,n-1}} \nu_{n-1}'
		= \sum_{n\geq 1} g(n) \frac{\cnd_{n-1,n-1}}{\thp_{n-1,n-1}}\frac{\cnd_{n-1,n}}{\cnd_{n-1,n-1}} \nu_{n-1}'	\\
		&= \sum_{n\geq 1} g(n) \frac{\thp_{n-1,n-1}}{\thp_{n-1,n-1}}\frac{\cnd_{n-1,n}}{\cnd_{n-1,n-1}} \nu_{n-1}
		= \sum_{n\geq 0} g(n+1) \frac{\cnd_{n,n+1}}{\cnd_{n,n}} \nu_{n}.
\end{align*}
Note that on the very left hand side, the sum extends to $\N_0$ with the convention that $\frac{\thp_{0,-1}}{\thp_{-1,-1}}=0$.
\end{proof}

It follows directly that under these conditions,
\begin{align}  \label{eq:gp:det-bal}
	\nu_n \cdot \frac{ \thp_{n,n-1} }{ \thp_{n-1,n-1} } = \nu_{n-1} \cdot \frac{ \cnd_{n-1,n} }{ \cnd_{n-1,n-1} } 
\end{align}
for all $n\in\N_0$ such that $\nu_n>0$. The latter is a reversibility condition for a birth-and-death chain with birth rates given by the quotient on the right hand side and death rates given by the quotient on the left hand side. This reversibility implies an analogue of Kolmogorov's cycle condition. Later it is shown that an alternating cycle condition on the thinning and condensation matrix suffice for existence of a (possibly infinite) measure $\nu$ solving the balance equation~\eqref{eq:gp:balance} and therefore is an existence criterion for $\nu$ without making use of $\nu$. The appearance of such a condition is not surprising since~\eqref{eq:gp:balance} is a balancing condition itself between $\nu$ and its thinning $\nu'$; and since two different transport mechanisms are used, they have to appear alternately.

\begin{theorem}[Alternating cycle condition for $\nu$]
Let $\nu$, $\thp$ and $\cnd$ be a distribution, a thinning matrix and a condensation matrix satisfying Equation~\eqref{eq:gp:balance}. Then for all $n,k,i,j\in\N_0$ such that $j\geq n$, $i\leq j,n$ and $\nu_n>0$,
\begin{align}	\label{eq:gp:cycle-cond}
	\thp_{ni} \cnd_{ij} \thp_{jk} \cnd_{kn} = \thp_{nk} \cnd_{kj} \thp_{ji} \cnd_{in}.
\end{align}
\end{theorem}

For other choices of the indices $n$, $k$, $i$ and $j$ both products vanish anyways.

\begin{proof}
Choose $i,j,n,k\in\N_0$ such that $i\leq j,n$ and $j\geq n$, then a repeated application of~\eqref{eq:gp:balance} yields
\begin{align*}
	\nu_n \thp_{n,i} \cnd_{i,j} \thp_{j,k} \cnd_{k,n} 
		&= \nu_i' \cnd_{i,n} \cnd_{i,j} \thp_{j,k} \cnd_{k,n}
		= \nu_j \thp_{j,i} \cnd_{i,n} \thp_{j,k} \cnd_{k,n}\\
		&= \nu_k' \cnd_{k,j} \thp_{j,i} \cnd_{i,n} \cnd_{k,n}
		= \nu_n  \thp_{n,k} \cnd_{k,j} \thp_{j,i} \cnd_{i,n}.
\end{align*}
Hence, the alternating cycle condition is fulfilled for all indices such that $\nu_n>0$.
\end{proof}

Of course, this condition also holds when starting with the condensation matrix $\cnd$ instead of the thinning matrix $\thp$. Then in the proof the distribution $\nu$ just needs to be replaced by its thinning $\nu'$. Both conditions will become more delicate in a non-discrete setting.

\begin{proposition}[Alternating cycle condition for $\nu'$]
Let $\nu$, $\thp$ and $\cnd$ be a distribution, a thinning matrix and a condensation matrix satisfying Equation~\eqref{eq:gp:balance}. Then for all $n,k,i,j\in\N_0$ such that $j\geq n$, $i\leq j,n$ and $\nu_i'>0$,
\begin{align}	\label{eq:gp:cycle-cond-prime}
	\cnd_{ij} \thp_{jk} \cnd_{kn} \thp_{ni} = \cnd_{in} \thp_{nk} \cnd_{kj} \thp_{ji}.
\end{align}
\end{proposition}

%%%%%%%%%%%%%
\subsection{Thinning characterization}	
%%%%%%%%%%%%%

%Go back to the balance equation~\eqref{eq:gp:balance} and assume that some distribution $\nu$ satisfies it together with a thinning matrix $P$ and an inference matrix $Q$. Start with a simple observation.
%
%\begin{lemma}	\label{thm:gp:const}
%Let $\nu$, $P$ and $Q$ satisfy Equation~\eqref{eq:gp:balance}. Then for each $k\in\N_0$, the mapping
%\begin{align*}
%	\{Q_{k,\scdot}>0\} \to \R, \qquad 
%	n \mapsto \nu_n \frac{P_{n,k}}{Q_{k,n}}
%\end{align*}
%is constant.
%\end{lemma}

A distribution $\nu$ and a suitable thinning matrix $\thp$ determine a unique condensation matrix $\cnd$. On the contrary, given a thinning matrix $\thp$ and a condensation matrix $\cnd$, a distribution $\nu$ satisfying the balance equations~\eqref{eq:gp:balance} (and the assumption in Theorem~\ref{thm:dn:ibpf}) should be unique thanks to Recursion~\eqref{eq:gp:det-bal}. But existence is not guaranteed; the system of linear equations is overdetermined. Essentially, if $\thp$ and $\cnd$ satisfy the alternating cycle condition~\eqref{eq:gp:cycle-cond}, then the existence of a measure $\nu$ is ensured, but it is possibly infinite.

Note that if $\thp$ is a positive thinning matrix, then also $\cnd$ is positive, at least up to some $n_0\in\N\cup\{\infty\}$, i.e. if there exists some $n_0$ such that $\cnd_{n,n_0}=0$ for some $n<n_0$, then there is no chance to get from any $k<n_0$ to any $m>n_0$, i.e. $\cnd_{k,m}=0$. 

\begin{theorem}	\label{thm:gp:inv}
Let $\thp$ be a positive thinning matrix and $\cnd$ be a condensation matrix which satisfy the alternating cycle condition~\eqref{eq:gp:cycle-cond}. Let $n_0$ be the smallest index such that $\cnd_{n,n+1}=0$. Then there exists a measure $\nu$ on $\N_0$ with support $\{0,\ldots,n_0\}$, which satisfies
\begin{align}	\label{eq:gp:thm-inv-balance}
	\nu_n \thp_{n,k} = \nu_k' \cnd_{k,n}
\end{align}
for all $n,k\leq n_0$. Moreover, this measure is determined recursively by
\begin{align}	\label{eq:gp:thm-inv-rec}
	\nu_n = \nu_{n-1} \cdot \frac{ \thp_{n-1,n-1} }{ \thp_{n,n-1} } \cdot \frac{ \cnd_{n-1,n} }{ \cnd_{n-1,n-1} },
		\qquad n\in\N.
\end{align}
\end{theorem}

\begin{proof}
Note first that any solution $\nu$ of~\eqref{eq:gp:thm-inv-balance} on $\N_0$ needs to satisfy~\eqref{eq:gp:thm-inv-rec} due to the integration-by-parts formula. Hence it suffices to show that if $\nu$ is defined recursively by~\eqref{eq:gp:thm-inv-rec}, then also~\eqref{eq:gp:thm-inv-balance} holds. Let $\nu$ be a sequence defined by~\eqref{eq:gp:thm-inv-rec} and $\nu_0=1$.

Since then $\nu'=\nu\thp$, the claim follows from
\begin{align}	\label{eq:gp:prf:2step}
	\nu_j \thp_{j,k} \cnd_{k,n} = \nu_n \thp_{n,k} \cnd_{k,j}
\end{align}
for all $j,k,n\leq n_0$ such that $k\leq n$, since then
\begin{align*}
	\nu_k' \cnd_{k,n} &= \sum_{j\geq k} \nu_j \thp_{j,k} \cnd_{k,n}
		= \nu_n \thp_{n,k} \sum_{j\geq k} \cnd_{k,j} 
		= \nu_n \thp_{n,k}.
\end{align*}

The first step is to show that $\nu$ also satisfies
\begin{align*}
	\nu_n \thp_{n,k}\cnd_{k,k}=\nu_k \thp_{k,k} \cnd_{k,n}
\end{align*}
for all $k<n\leq n_0$, or, equivalently,
\begin{align}	\label{eq:gp:prf:cycle}
	\nu_n &= \nu_{k} \frac{ \thp_{k,k} \cnd_{k,n} }{ \cnd_{k,k} \thp_{n,k} }.
\end{align}
Note first that for $k=n-1$ this is~\eqref{eq:gp:thm-inv-rec} rewritten. Assume that the claim holds for some $k<n$ and show that this holds also for $k-1$ as long as $k>0$. Indeed, applying the induction assumption, reordering, applying the induction initial step and the alternating cycle condition,
\begin{multline*}
	\nu_n \thp_{n,k-1} \cnd_{k-1,k-1} \thp_{n,k} \cnd_{k,k} \thp_{k,k-1}\\
	\begin{aligned}
		&= \nu_k \thp_{k,k}\cnd_{k,n} \thp_{n,k-1}\cnd_{k-1,k-1} \thp_{k,k-1}
		= \nu_k \thp_{k,k-1}\cnd_{k-1,k-1} \thp_{k,k}\cnd_{k,n} \thp_{n,k-1}\\
		&= \nu_{k-1} \thp_{k-1,k-1}\cnd_{k-1,k} \thp_{k,k}\cnd_{k,n} \thp_{n,k-1}
		= \nu_{k-1} \thp_{k-1,k-1} \cnd_{k-1,n} \thp_{n,k}\cnd_{k,k} \thp_{k,k-1},
	\end{aligned}
\end{multline*}
which is the first step ($\thp_{n,k} \cnd_{k,k} \thp_{k,k-1}>0$ by assumption). 
%Note that this reads in terms of integrals
%\begin{align*}
%	\sum_{n\geq k} g_n \nu_n P_{n,k}Q_{k,k}
%		= \nu_k P_{k,k} \sum_{n\geq k} Q_{k,n} g_n,
%\end{align*}
%or equivalently,
%\begin{align*}	
%	\sum_{k\geq 0} f_k \sum_{n\geq k} g_n \nu_n P_{n,k}Q_{k,k}
%		= \sum_{k\geq 0} f_k \nu_k P_{k,k} \sum_{n\geq k} Q_{k,n} g_n.
%\end{align*}

Secondly, assume $\thp_{j,n} \cnd_{n,n}>0$ for $j\geq n$ and $\thp_{j,j}\cnd_{j,n}>0$ for $j<n$, then by~\eqref{eq:gp:prf:cycle} and the alternating cycle condition, for $j\geq n$,
\begin{align*}
	\nu_j \thp_{j,k}\cnd_{k,n} \thp_{j,n}\cnd_{n,n}
		= \nu_n \thp_{n,n}\cnd_{n,j} \thp_{j,k}\cnd_{k,n}
		= \nu_n \thp_{n,k}\cnd_{k,j} \thp_{j,n}\cnd_{n,n}.
\end{align*}
Similarly, for $j<n$,
\begin{align*}
	\nu_j \thp_{j,k}\cnd_{k,n} \thp_{j,j}\cnd_{j,n}
		= \nu_n \thp_{n,j}\cnd_{j,j} \thp_{j,k}\cnd_{k,n}
		= \nu_n \thp_{n,k}\cnd_{k,j} \thp_{j,j}\cnd_{j,n},
\end{align*}
which proves~\eqref{eq:gp:prf:2step} and thus completes the proof.
\end{proof}

%\begin{remark}
%The outlined procedure requires that $\thp$ as well as $\cnd$ have positive entries on the diagonal and below, respectively above, which implies that whenever $\nu$ vanishes at some $k$, it also vanishes for all $n>k$. However, this property may be relaxed -- but then there are situations in which the recursion~\eqref{eq:gp:thm-inv-rec} does no longer hold and some of the cycles of length four are not valid cycles. In this such situations, cycle conditions for longer cycles have to be taken into account, and possible zeros in the proof replaced by longer cycles.
%\end{remark}

\begin{remark}
\begin{enumerate}
	\item Theorem~\ref{thm:gp:inv} does not ensure that the measure $\nu$ is finite. Start with~\eqref{eq:gp:balance} again and sum over $k$, then, starting from the right side,
\begin{align*}
	\nu_n = \sum_k \nu_k'\cnd_{k,n} = \sum_k \sum_j \nu_j \thp_{j,k}\cnd_{k,n} = (\nu \thp\cnd)_n.
\end{align*}
Hence $\nu$ is an invariant measure for the (full) matrix $A\defeq \thp\cnd$, and thus $\nu$ is finite if and only if the stochastic matrix $A$ is the transition matrix of a positive recurrent Markov chain on $\N_0$.
	\item Observe that $A=\thp\cnd$ is in fact a LU-decomposition of $A$; a decomposition into a product of a lower and an upper triangular matrix, and thus the two considered problems may be reformulated as:
		\begin{enumerate}
			\item For a given distribution $\nu$ and a lower triangular, stochastic matrix $\thp$ determine a stochastic matrix $\cnd$ such that the product $\thp\cnd$ has a left eigenvector $\nu$ for the eigenvalue 1 (or $\nu$ is an invariant distribution for the product).
			\item For given lower and upper triangular, stochastic matrices $\thp$ and $\cnd$ find a vector $\nu$, which is an invariant distribution for $\thp\cnd$.
		\end{enumerate}
\end{enumerate}
\end{remark}

%%%%%%%%%%%%%
\subsection{Examples}	\label{sect:gp:ex}
%%%%%%%%%%%%%

In this section, several objects that occurred previously are given explicitly for the independent as well as the uniform thinning. Particularly the independent thinning is intimately linked to the Panjer class of distributions, which consists the Poisson, binomial and negative binomial distributions~\cite{hP81,SJ81}. For the independent $q$-thinning, the quotient of the thinning matrix 
\begin{align*}
	\frac{\thp_{n,n-1}}{\thp_{n-1,n-1}} = 1-q
\end{align*}
is consant.

\begin{proposition}
Assume that $\thp$ is an independent $q$-thinning for some $q\in(0,1)$ and let $\cnd$ be a condensation matrix. Then
\begin{align*}
	(1-q) n\nu_n = \nu_{n-1} \cdot \frac{ \cnd_{n-1,n} }{ \cnd_{n-1,n-1} }.
\end{align*}
Particularly, $\nu$ satisfies an integration-by-parts formula with
\begin{align*}
	\pi(n) = \frac{ \cnd_{n,n+1} }{ (1-q) \cnd_{n,n} }, \qquad n\in\N_0.
\end{align*}
\end{proposition}
The following recursions are well-known.
\begin{example}	\label{ex:gp_panjer}
\begin{enumerate}
	\item For the Poisson distribution with parameter $\lambda$, $\cnd_{n-1,n-1} = \e^{-(1-q)\lambda}$ and $\cnd_{n-1,n} = (1-q)\lambda \e^{-(1-q)\lambda}$, and therefore
		\begin{align*}
			n\nu_n = \nu_{n-1} \cdot \lambda
		\end{align*}
	\item For the binomial distribution with parameters $r$ and $p$, $\cnd_{n-1,n-1} = \left(\frac{1-p}{1-pq}\right)^{r-(n-1)}$ and $\cnd_{n-1,n} = \binom{r-(n-1)}{1} \left(\frac{p(1-q)}{1-pq}\right)^1 \left(\frac{1-p}{1-pq}\right)^{r-(n-1)-1}$, and therefore
		\begin{align*}
			n\nu_n &= \nu_{n-1} \cdot \frac{ (r-n+1) \left(\frac{p(1-q)}{1-pq}\right)^1 \left(\frac{1-p}{1-pq}\right)^{r-(n-1)-1} }{ (1-q) \left(\frac{1-p}{1-pq}\right)^{r-(n-1)} }
%				= \nu_{n-1} \cdot \frac{ (r-n+1) \frac{p(1-q)}{1-pq} }{ (1-q) \frac{1-p}{1-pq} }\\
				= \nu_{n-1} \cdot (r-n+1) \frac{p}{1-p}
		\end{align*}
		as expected.
	\item For the negative binomial distribution with parameters $r$ and $p$, the matrix $\cnd$ has $\cnd_{n-1,n-1} = \bigl(1-(1-q)p\bigr)^{r+(n-1)}$ and $\cnd_{n-1,n} = \binom{r+(n-1)}{1} p(1-q) \bigl(1-(1-q)p\bigr)^{r+(n-1)}$, and hence
		\begin{align*}
			n\nu_n = \nu_{n-1} \cdot \frac{ (r+n-1) p(1-q) }{ 1-q }
				= \nu_{n-1} \cdot (r+n-1) p.
		\end{align*}
\end{enumerate}
\end{example}

Note that in each of these cases, it is quite convenient to deal with the splitting matrix $\Upsilon$ instead of $\cnd$. Since $\Upsilon_{k,\scdot}$ is of the same type, it also satisfies an integration-by-parts formula.

For the discrete stick breaking, the recursion still takes a simple form. Polynomial laws stay polynomial under the uniform thinning.

\begin{proposition}
Assume that $\thp$ is a uniform thinning, then 
\begin{align*}
	\frac{ n }{ n+1 } \cdot \nu_n = \nu_{n-1} \cdot \frac{ \cnd_{n-1,n} }{ \cnd_{n-1,n-1} }.
\end{align*}
Particularly,
\begin{align*}
	\nu_n = (n+1) \nu_0 \cdot \prod_{j=0}^{n-1} \frac{ \cnd_{j,j+1} }{ \cnd_{j,j} }
\end{align*}
\end{proposition}
Immediately, convergence of the series in the next line is equivalent to $\nu$ being a finite measure,
\begin{align*}
	1 = \nu_0 \left( 1 + \sum_{k=0}^\infty \prod_{j=0}^{k} (k+2)\frac{ \cnd_{j,j+1} }{ \cnd_{j,j} } \right).
\end{align*}

\begin{example}	\label{ex:gp_power}
Denote the Hurwitz zeta function by
\begin{align*}
	\zeta_{\alpha,q} \defeq \sum_{j\geq 0} \frac{1}{(j+q)^\alpha}.
\end{align*}
%\begin{enumerate}
%	\item 
Let $\nu_n=\zeta_{\alpha,1}^{-1}(n+1)^{-\alpha}$, $n\in\N_0$, for some $\alpha>1$. Then 
		\begin{align*}
			\nu_k'= \frac{ 1 }{ \zeta_{\alpha,1} } \sum_{j\geq k} (j+1)^{-(\alpha+1)}
				= \frac{ \zeta_{\alpha+1,k+1} }{ \zeta_{\alpha,1} }
		\end{align*}
		and consequently, the condensation matrix is given by
		\begin{align*}
			\cnd_{k,n} = \frac{1}{\zeta_{\alpha,1}(n+1)^\alpha} \cdot \frac{1}{n+1} \cdot \frac{ \zeta_{\alpha,1} }{ \sum_{j\geq k} (j+1)^{-(\alpha+1)} } 
				= \frac{ 1 }{ (n+1)^{\alpha+1} } \cdot \frac{ 1 }{ \zeta_{\alpha+1,k+1} } 
		\end{align*}
		for $n\geq k$, such that
		\begin{align*}
			\frac{ \cnd_{n-1,n} }{ \cnd_{n-1,n-1} } = \frac{ \zeta_{\alpha+1,n} n^{\alpha+1} }{ \zeta_{\alpha+1,n} (n+1)^{\alpha+1} }
				= \left( \frac{ n }{ n+1 } \right)^{\alpha+1}.
		\end{align*}				
		Equivalently, the splitting matrix is given for all $k,l\geq 0$ by
		\begin{align*}
			\Upsilon_{k,l} = \frac{ 1 }{ \zeta_{\alpha+1,k+1} (k+l+1)^{\alpha+1} },
		\end{align*}
		which in terms of stick breaking is the probability observing a remainder of length $l$ given the first part has length $k$.
%\end{enumerate}
\end{example}

\begin{example}
\begin{enumerate}
	\item Let $\thp$ be the all-or-nothing thinning matrix for success probability $q\in(0,1)$ and let $\nu$ be any distribution on $\N_0$. Then the thinned distribution $\nu'$ is given by
		\begin{align*}
			\nu_0' = (1-q)+q\nu_0,\qquad \nu_n' = q\nu_n,\quad n\geq 1.
		\end{align*}
		Hence, if something larger then zero is observed, that is the perfect guess; otherwise one should toss a coin with success probability $q$ and predict 0 in case of success, and any number according to $\nu$ (including 0) otherwise, i.e. according to~\eqref{eq:gp:balance},
		\begin{align*}
			\cnd_{k,n} = \begin{cases}
				\nu_n\cdot q\cdot \frac{1}{q\nu_n}=1, & n=k;\\
				\frac{\nu_0}{1-q+q\nu_0}, & n=k=0;\\
				\frac{(1-q)\nu_n}{1-q+q\nu_0}, & k=0,\, n>0;\\
				0, & \text{otherwise.}
			\end{cases}
		\end{align*}
	\item Let $\thp$ be the taking-almost-nothing thinning matrix for some $q\in(0,1)$ and let $\nu$ be any distribution on $\N_0$. $\nu'$ is given by
		\begin{align*}
			\nu_0' = \nu_0 + (1-q)\nu_1, \qquad \nu_n'= q\nu_n + (1-q)\nu_{n+1}, \quad n\geq 1.
		\end{align*}
		Therefore,
		\begin{align*}
			\cnd_{k,n} = \begin{cases}
				\nu_n\cdot q\cdot \frac{1}{q\nu_n+(1-q)\nu_{n+1}} = \frac{q\nu_n}{q\nu_n+(1-q)\nu_{n+1}}, & n=k\geq 1;\\
				\frac{(1-q)\nu_{n}}{q\nu_{n-1}+(1-q)\nu_{n}}, & n=k+1\geq 2;\\
				\frac{\nu_0}{\nu_0+(1-q)\nu_1}, & n=k=0;\\
				\frac{(1-q)\nu_1}{\nu_0+(1-q)\nu_1}, & k=0,\, n=1;\\
				0, & \text{otherwise.}
			\end{cases}
		\end{align*}
		Note that the quotient $\nicefrac{\cnd_{n-1,n}}{\cnd_{n-1,n-1}}=\nicefrac{(1-q)\nu_n}{q\nu_{n-1}}$ essentialy contains all information about $\nu$.
\end{enumerate}
\end{example}

%%%%%%%%%%%%%%%%%
\section{Finite point processes and general splittings}	\label{sect:fpp}
%%%%%%%%%%%%%%%%%
\subsection{Finite point processes and thinnings}
%%%%%%%%%%%%%

The basic results carry over to the case of point processes with more technicalities. Along that way, the splitting kernel is identified as being absolutely continuous with respect to the Palm kernel, which sheds new light on a result in~\cite{aK86}. A second ingredient relates thinning kernels to the Campbell measure of point processes and thus links to integration-by-parts formulas. Again, an alternating cycle condition on a thinning and a condensing kernel together with a summability condition yields a characterization of a point process by means of these two kernels. At this point it should be remarked that if the space, on which the original distribution operates, is discrete, the proofs simplify. Nevertheless they are instructive even in the discrete situation.

Throughout this section, let $\pp$ be a finite point process on a Polish space $X$, i.e. a distribution on the set of finite point measures $\Mpf(X)$ (when equipped with the $\sigma$-algebra generated by all evaluation mappings $\zeta_B:\mu\mapsto\mu(B)$ for Borel sets $B\subseteq X$, $\Mpf(X)$ is Polish itself), see e.g.~\cite{DVJ03,oK83,kK14}. Since in general there are no ambiguities about the underlying space, it is also denoted by $\Mpf$. Any $\mu\in\Mpf$ should be understood as a collection of points of $X$ (with possible multiplicities). $\mu$ can be written as
\begin{align*}
	\mu = \sum_{i} \delta_{x_i}
\end{align*}
where the sum is finite and the $x_i$ are (not necessarily mutually distinct) locations of the points. The term $x\in\mu$ is understood as the point measure $\mu$ contains an atom at $x$, and $\mu-\delta_x$ is understood as the point measure $\mu$ with and atom at $x$ removed.

\begin{definition}
A thinning kernel on $X$ is a stochastic kernel
\begin{gather*}
	\thp: \Mpf(X) \to \Mpf(X), \qquad \mu \mapsto \thp_\mu(\d\eta)
\end{gather*}
such that $T_\mu$ is concentrated on $\Mpm_\mu \defeq \bigl\{\eta\in\Mpf(X): \eta\leq\mu \bigr\}$.
\end{definition}

For a finite point process $\pp$ on $X$, denote by $\pp'\defeq \thp\pp$ the thinning of $\pp$ with respect to $\thp$, i.e. the image of $\pp$ under $\thp$. $\pp'$ is a finite point process which realizes subconfigurations of the point configurations of $\pp$. Denote their joint law by $\thm$. The existence of a condensation kernel is ensured by an abstract disintegration theorem applied to $\thm$~\cite{oK84,bN13a}.
\begin{lemma}
Let $\pp$ be a finite point process and $\thp$ be a thinning kernel on $X$. Then there exists a stochastic kernel
\begin{gather*}
	\cnd: \Mpf(X) \to \Mpf(X), \qquad \eta \mapsto Q_\eta(\d\mu),
\end{gather*}
the \emph{condensation kernel}, which is concentrated on $\{\mu\in\Mpf(X): \mu\geq\eta\}$, such that
\begin{align}	\label{eq:fpp:cond-disint}
	\int g(\eta,\mu) \thm(\d\mu,\d\eta) = \iint g(\eta,\mu) \thp_\mu(\d\eta) \pp(\d\mu) = \iint g(\eta,\mu) \cnd_\eta(\d\mu) \pp'(\d\eta)
\end{align}
for all non-negative, measurable functions $g$.
\end{lemma}

Note that the splitting measure $\spm$ is related to $\thm$ via
\begin{align*}
	\int g(\eta,\mu) \spm(\d\mu,\d\eta) = \int g(\eta,\mu-\eta) \thm(\d\mu,\d\eta),
\end{align*}
and hence its disintegration with respect to $\pp'$, the splitting kernel $\Upsilon$, to $\cnd$ via
\begin{align*}
	\int \phi(\mu) \cnd_\eta(\d\mu) = \int \phi(\nu+\eta) \Upsilon_\eta(\d\nu)
\end{align*}
for all non-negative, measurable $\phi$ and $\pp'$-a.e. $\eta$.

Moreover, for each $\mu\in\Mpf$,  $\thp_\mu$ is in fact a discrete distribution describing how to sample from a (finite) population $\mu$ distributed in space. In contrast to that, the condensation kernel is typically not discrete.

For a point measure $\mu$ denote by $\mu^{-[n]}$ its $n$-th falling factorial measure, i.e.
\begin{align*}
	\int g(x_1,\ldots,x_n)\mu^{-[n]}(\d x_1,\ldots,\d x_n) = \int g(x_1,\ldots,x_n) \bigl(\mu-\delta_{x_{n-1}}-\ldots-\delta_{x_1}\bigr)(\d x_n)\cdots \mu(\d x_1),
\end{align*}
which is the sum over all ordered subsets of $\mu$ of size $n$ if $\mu$ has no multiple points. A general thinning can be written as
\begin{align*}
	T_\mu(\phi) = Z_\mu \sum_{n\geq 0} \frac{1}{n!} \int \phi(\delta_{x_1}+\ldots+\delta_{x_n}) t(\mu;x_1,\ldots,x_n) \mu^{-[n]}(\d x_1,\ldots,\d x_n)
\end{align*}
for some weight function $t$ for the subconfiguration built from $x_1,\ldots,x_n$, which is symmetric in $x_1,\ldots,x_n$. The first sumand is read as being a non-negative real. Also, write $\vx=(x_1,\ldots,x_n)$ and $\delta_{\vx}$ instead of $\delta_{x_1}+\ldots+\delta_{x_n}$ (point configurations on products of $X$ are not considered here). For later convenience, $Z_\mu$ is the (inverse of the) normalization. 

Two generalizations of the thinnings in Section~\ref{sect:trv} shall be studied here, referred to as type 1 and 2:
\begin{enumerate}
	\item similar to~\cite{bN13a}, the weight function depends on the sample, but not on the set $\mu$ from which the points are sampled, i.e. $t$ is a symmetric function of $\vx$ only;
	\item similar to mixed sample processes, $t$ depends on $\mu$ and $\vx$ only via their total mass or length, respectively, i.e. denoted briefly by $t(\mu;x_1,\ldots,x_n)=t_{|\mu|,n}$, where $|\mu|\defeq \mu(X)$ is the total mass of $\mu$. Given the number of sampled points the configuration is sampled uniformly.
\end{enumerate}
Note that the independent thinning is of both types with
\begin{align*}
	t(x_1,\ldots,x_n) = \left(\frac{q}{1-q}\right)^n,
\end{align*}
see~\cite[Lemma 6.3.2]{bN13a}. Contrarily, a ``uniform'' thinning necessarily depends on the number of points sampled from, see the examples in~\ref{sect:fpp:ex}

The following theorem generalizes the result in~\cite{aK86} that identifies the splitting kernel as being absolutely continuous with respect to the reduced Palm distributions of the given point process. The reduced Palm distributions appear as disintegrations of higher order reduced Campbell measures with respect to the factorial moment measure of $\pp$, and this is tailored for thinnings of the given type. In the case of an independent $q$-thinning, $Z_\mu=(1-q)^{|\mu|}$.

\begin{theorem}	\label{thm:fpp:split-palm}
The splitting kernel of a finite point process $\pp$ with existing moment measures of all orders related to a thinning $\thp$ of first type is given by
\begin{align*}
	\Upsilon_\nu(\phi) = \frac{1}{\int Z_{\eta+\nu} \pp^!_\nu(\d\eta)} \int \phi(\eta) Z_{\eta+\nu} \pp^!_\nu(\d\eta)
\end{align*}
for all non-negative, measurable functions $\phi$ and $\pp'$-a.e. $\nu$.
\end{theorem}

\begin{proof}
The spirit of the proof agrees with~\cite[Proposition 6.3.5]{bN13a} apart from the particular thinning. Firstly, observe that the factorials of $\mu$ together with $\pp$ form reduced Campbell measures, which disintegrate into the factorial moment measures and reduced Palm measures,
\begin{align*}
	\spm(h) &= \sum_{n\geq 0} \frac{1}{n!} \iint h(\delta_\vx,\mu-\delta_\vx) \, Z_\mu \, t(\vx) \, \mu^{-[n]}(\d \vx) \pp(\d\mu)	\\
		&= \sum_{n\geq 0} \frac{1}{n!} \int h(\delta_\vx,\eta) \, Z_{\eta+\delta_\vx} \, t(\vx) \, C^{!,n}(\d\vx;\d\eta)	\\
		&= \sum_{n\geq 0} \frac{1}{n!} \iint h(\delta_\vx,\eta) \, Z_{\eta+\delta_\vx} \, t(\vx) \, \pp^!_\vx(\d\eta) \nu^{[n]}_\pp(\d\vx),
\intertext{where $\nu^{[n]}_\pp$ is the $n$-th factorial moment measure of $\pp$. Replacing this by its full integral and collecting terms for $\thp$, we continue as}
		&= \int \sum_{n\geq 0} \frac{1}{n!} \iint h(\delta_\vx,\eta) \, Z_{\eta+\delta_\vx} \, t(\vx) \pp^!_\vx(\d\eta) \mu^{-[n]}(\d\vx) \pp(\d\mu)	\\
		&= \iiint h(\nu,\eta) \, \frac{Z_{\eta+\nu}}{Z_\mu} \, \pp^!_\nu(\d\eta) \thp_\mu(\d\nu) \pp(\d\mu).
\end{align*}
If $h=h_1 \otimes h_2$, then
\begin{align*}
	\spm(h) = \iint \frac{ h_1(\nu) }{ Z_\mu }  \int h_2(\eta) Z_{\eta+\nu} \pp^!_\nu(\d\eta) \thp_\mu(\d\nu) \pp(\d\mu).
\end{align*}
The inner integral is, apart from the normalization, the splitting kernel; the missing constant is obtained by setting $h_2\equiv 1$.
\end{proof}

Thus, given a partial observation $\nu$ of a point configuration realized by $\pp$, the distribution of the remainder is essentially given by its reduced Palm distribution $\pp^!_\nu$. The distribution of the entire configuration follows immediately. 

\begin{corollary}	\label{thm:fpp:cnd-palm}
For $\pp'$-a.e. $\nu\in\Mpf(X)$,
\begin{align*}
	\cnd_\nu(\phi) = \frac{1}{\pp^!_\nu(Z_{\nu+\scdot})} \int \phi(\nu+\eta) Z_{\nu+\eta} \pp^!_\nu(\d\eta)
\end{align*}
for any non-negative, measurable function $\phi$. If in addition $\nu(X)=n$, then
\begin{align*}
	\cnd_\nu(\zeta_X=n) = \frac{Z_\nu \pp^!_\nu(\zeta_X=0)}{\pp^!_\nu(Z_{\nu+\scdot})}.
\end{align*}
\end{corollary}

Thus, $\cnd_\nu$ has an atom at $\nu$ if and only if the reduced Palm distribution realizes the empty configuration with positive probability. Note that under $\cnd_\nu$, if $\nu(X)=n$, the event $\{\zeta_X=n\}$ contains $\nu$ only.

\begin{proof}
The first statement follows immediately from Theorem~\ref{thm:fpp:split-palm}, the second from the observation
\begin{equation*}
	\cnd_\nu(\zeta_X=n) = \frac{1}{\pp^!_\nu(Z_{\nu+\scdot})} \int 1_{\eta(X)=0} Z_{\nu+\eta} \pp^!_\nu(\d\eta)
		= \frac{Z_\nu}{\pp^!_\nu(Z_{\nu+\scdot})} \int 1_{\eta(X)=0} \pp^!_\nu(\d\eta).	\qedhere
\end{equation*}
\end{proof}

\begin{remark}
Observe that for a thinning of the second type, any occurrence of $Z_{\eta+\nu}$ is replaced by $Z_{\eta+\nu}t_{|\eta+\nu|, |\nu|}$, and any occurrence of $Z_\mu$ by $Z_\mu t_{|\mu|, |\nu|}$ assuming that these weights are positive.
\end{remark}

%%%%%%%%%%%%%
\subsection{Integration-by-parts}
%%%%%%%%%%%%%

The aim is to show how a finite point process and a suitable thinning kernel $\thp$ yield an integration-by-parts formula for $\pp$. In the case of random variables, the essential property is that the thinning admits to sample the entire population as well as the entire population with an individual removed with positive probability. The analogue assumption on $\thp$ is made throughout this discussion.

Suppose that $\thp$ is such that for all $\mu\in\Mpf$,
\begin{align*}
	\thp_\mu(\{\mu\})>0,	\qquad
	\int \thp_\mu(\mu-\delta_x) \mu(\d x) > 0,
\end{align*}
i.e. the thinning is allowed to remove no point or at least one single point of a configuration. In such a case and given $\pp$, at least $\cnd_\eta(\{\eta\})>0$ for $\pp$-a.e. $\eta$. In contrast, adding a particular point to a configuration does not need to have a positive probability, but $\cnd(\zeta_X=|\eta|+1)>0$.

To prepare the the main result of this part, the integration-by-parts formula, sampling $n-1$ points from a collection of $n$ points is related to sampling exactly one point. Of particular interest is which weight each single point gets during this procedure. For $n\in\N_0$ denote by $\Mpm_n(X)\subset\Mpf(X)$ the set of point measures with $n$ points (including possible multiplicities).

%First we want to find an analogue of $\nu_n \frac{P_{n,n-1}}{P_{n-1,n-1}} = \nu_{n-1} \frac{Q_{n-1,n}}{Q_{n-1,n-1}}$. Since the analogue of the fraction on the left is discrete, that one is obtained immediately,
%\begin{align*}
%	\frac{T_{\mu}(\mu-\delta_{x})}{T_{\mu-\delta_{x}}(\mu-\delta_{x})} 
%		&= \frac{ \binom{\mu(x)}{\mu(x)-1} q^{\mu(x)-1}(1-q) }{ q^{\mu(x)-1}}
%		= (1-q)\mu(x), \qquad x\in\mu.
%\end{align*}
%This is a strong hint towards partial integration. The other quotient cannot be computed that way since it is not discrete in general. But what should be expected is, if $P$ is Papangelou, then the quotient is the Papangelou kernel initialized with the given realization of the subsample.

%Recall the alternative proof:
%\begin{proof}[Alternative proof]
%Let $g:\N_0\to\R_+$ be bounded, then by a double application of~\eqref{eq:gp:balance} and an index shift,
%\begin{align*}
%	\sum_{n\geq 1} g(n) \frac{P_{n,n-1}}{P_{n-1,n-1}} \nu_n 
%		&= \sum_{n\geq 1} g(n) \frac{Q_{n-1,n}}{P_{n-1,n-1}} \nu_{n-1}'
%		= \sum_{n\geq 1} g(n) \frac{Q_{n-1,n-1}}{P_{n-1,n-1}}\frac{Q_{n-1,n}}{Q_{n-1,n-1}} \nu_{n-1}'	\\
%		&= \sum_{n\geq 1} g(n) \frac{P_{n-1,n-1}}{P_{n-1,n-1}}\frac{Q_{n-1,n}}{Q_{n-1,n-1}} \nu_{n-1}	
%		= \sum_{n\geq 0} g(n+1) \frac{Q_{n,n+1}}{Q_{n,n}} \nu_{n}.
%\end{align*}
%Note that on the very left hand side, the sum extends to $\N_0$ with the convention that $\frac{P_{0,-1}}{P_{-1,-1}}=0$.
%\end{proof}

\begin{lemma}	\label{thm:ibpf:thn}
Let $\mu\in\Mpm_n(X)$ and $\thp$ be a thinning kernel of type 1. Then for any non-negative, measurable function $\phi$,
\begin{align*}
	\thp_\mu \bigl( \phi 1_{\zeta_X=n-1} \bigr) = \int \frac{Z_\mu}{Z_{\mu-\delta_x}} \thp_{\mu-\delta_x}\bigl( \phi 1_{\zeta_X=n-1} \bigr) \mu(\d x).
\end{align*}
\end{lemma}

\begin{proof}
Starting from representation of the thinning, note that choosing $n-1$ elements out of $n$ is the same as choosing a single element out of $n$, i.e.
\begin{align*}
	\thp_\mu \bigl( \phi 1_{\zeta_X=n-1} \bigr) 
		&= \frac{Z_\mu}{(n-1)!} \int \phi( \delta_{\vy}) t(\vy) \mu^{-[n-1]}(\d\vy)	\\
		&= \frac{Z_\mu}{(n-1)!} \iint \phi( \delta_{\vy}) t(\vy) \bigl(\mu-\delta_x\bigr)^{-[n-1]} (\d\vy) \mu(\d x)	\\
		&= \int \frac{Z_\mu}{Z_{\mu-\delta_x}} \cdot \frac{Z_{\mu-\delta_x}}{(n-1)!} \int \phi( \delta_{\vy}) t(\vy) \bigl(\mu-\delta_x\bigr)^{-[n-1]} (\d\vy) \mu(\d x)	\\
		&= \int \frac{Z_\mu}{Z_{\mu-\delta_x}} \thp_{\mu-\delta_x} \bigl( \phi 1_{\zeta_X=n-1} \bigr) \mu(\d x). \qedhere
\end{align*}
\end{proof}

In particular for $\phi\equiv 1$ the lemma states
\begin{align*}
	T_\mu \bigl(\zeta_X=n-1\bigr) 
		= \int \frac{Z_\mu}{Z_{\mu-\delta_x}} T_{\mu-\delta_x}\bigl( \zeta_X=n-1 \bigr) \mu(\d x).
\end{align*}
This way a sum over all points of a point configuration enters the game, and the weight shall be given by the quotient of the normalization constants. A similar argument for $\cnd$ cannot be expected due to its typically non-discrete nature. Nevertheless, the strategy will be to extract how $\cnd$ adds a single point to a point configuration.

\begin{theorem}[Integration-by-parts formula fo $\pp$]
Let $\pp$ be a finite point process on $X$ such that if $\pp(\Mp_{n_0})=0$ for some $n_0\in\N$, then $\pp(\Mp_n)=0$ for all $n>n_0$, and $\thp$ a thinning kernel of type 1 such that for all $n\in\N$ and $\pp$-a.e. $\mu\in\Mpm_n(X)$,
\begin{align*}
	\thp_\mu\bigl(\{\mu\}\bigr) >0
	\quad\text{ and }\quad
	\int \thp_\mu\bigl(\mu-\delta_x\bigr)\mu(\d x) >0.
\end{align*}
Then for all non-negative, measurable functions $g$,
\begin{align}	\label{eq:fpp:thm-ibpf:ibpf}
	\iint g(x,\mu) \frac{Z_\mu}{Z_{\mu-\delta_x}}\mu(\d x) \pp(\d\mu) 
		= \iint g(x,\mu+\delta_x) \pi(\mu,\d x) \pp(\d\mu),
\end{align}
where the Papangelou kernel $\pi$ is given by
\begin{align*}
	\pi(\mu,\d x) = \frac{Z_{\mu+\delta_x}}{Z_\mu \pp^!_\mu(\zeta_X=0)} \rho(\mu,\d x)
\end{align*}
and $\rho(\mu,\scdot)$ is the measure on $X$ induced by $1_{\eta(X)=1}\pp^!_\mu(\d\eta)$.
\end{theorem}

Apart from the quotient, the left hand side in Equation~\eqref{eq:fpp:thm-ibpf:ibpf} is the Campbell measure of the point process $\pp$. In the case of the independent $q$-thinning, this quotient is constant and~\eqref{eq:fpp:thm-ibpf:ibpf} is a classical integration-by-parts formula. More comments are given in Section~\ref{sect:fpp:ex}.

\begin{proof}
Subsequently, denote by $\mu^\ast\subset X$ the support of a point measure $\mu$, and identify $\delta_x^\ast=\{x\}$ with the element $x\in X$. Let $g$ be a non-negative function, then by Lemma~\ref{thm:ibpf:thn} and (with the short notation for the points $y_1,\ldots,y_{n-1}$) $\thp_\mu(\zeta_X=n-1) = Z_\mu \int t(\vy) \bigl(\mu-\delta_x\bigr)^{-[n-1]}(\d \vy)\mu(\d x)>0$ whenever $\mu(X)=n$, and
\begin{multline*}
	\iint g(x,\mu) \frac{Z_\mu}{Z_{\mu-\delta_x}} \mu(\d x) \pp(\d\mu)\\
		\begin{aligned}
		&= \int \sum_{n\geq 1} 1_{\Mpm_n(X)}(\mu) \int g(x,\mu) \frac{\thp_{\mu-\delta_x}\bigl( \zeta_X=n-1 \bigr)}{\thp_{\mu-\delta_x}\bigl( \zeta_X=n-1 \bigr)} \mu(\d x) \pp(\d\mu) \\
		&= \int \sum_{n\geq 1} 1_{\Mpm_n(X)}(\mu) \int g\bigl((\mu-\nu)^\ast,\mu) \frac{1_{\Mpm_{n-1}(X)}(\nu)}{\thp_{\nu}\bigl( \nu \bigr)} \thp_\mu(\d\nu) \pp(\d\mu) \\
		&= \int \sum_{n\geq 1} \int 1_{\Mpm_n(X)}(\mu) g\bigl((\mu-\nu)^\ast,\mu) \frac{1_{\Mpm_{n-1}(X)}(\nu)}{\thp_{\nu}\bigl( \nu \bigr)} \cnd_\nu(\d\mu) \pp'(\d\nu)
		\end{aligned}
\end{multline*}
Consider the inner integral, for which $\nu(X)=n-1$ is fixed, with $\thp_{\nu}\bigl( \nu \bigr)$ replaced by $\cnd_{\nu}\bigl( \nu \bigr)$, then by Corollary~\ref{thm:fpp:cnd-palm},
\begin{multline*}
	\int 1_{\Mpm_n(X)}(\mu) g\bigl((\mu-\nu)^\ast,\mu) \frac{1}{\cnd_{\nu}\bigl( \nu \bigr)} \cnd_\nu(\d\mu)\\
		\begin{aligned}
			&= \frac{1}{\pp^!_\nu(Z_{\nu+\scdot})} \int 1_{\Mpm_1(X)}(\eta) g(\eta^\ast,\nu+\eta) \frac{Z_{\nu+\eta}}{\cnd_{\nu}\bigl( \nu \bigr)} \pp^!_\nu(\d\eta)\\
			&= \frac{1}{\pp^!_\nu(\zeta_X=0)} \int 1_{\Mpm_1(X)}(\eta) g(\eta^\ast,\nu+\eta) \frac{Z_{\nu+\eta}}{Z_\nu} \pp^!_\nu(\d\eta) \eqdef \phi(\nu),
		\end{aligned}
\end{multline*}
where for the last equality $\cnd_\nu(\nu)\pp_\nu^!(Z_{\nu+\scdot}) = Z_\nu \pp_\nu^!(\zeta_X=0)$ is used. Observe that the last line yields exactly the integration with respect to $\pi$. Thus, continuing the evaluation of the entire integral yields by switching the integrals once more
\begin{multline*}
	\sum_{n\geq 1} \int \frac{1_{\Mpm_{n-1}(X)}(\nu)}{\thp_{\nu}\bigl( \nu \bigr)} \phi(\nu) \cnd_\nu(\nu) \pp'(\d\nu)\\
		\begin{aligned}
			&= \sum_{n\geq 0} \iint \frac{1_{\Mpm_n(X)}(\nu)}{\thp_{\nu}\bigl( \nu \bigr)} \phi(\nu) 1_{\Mpm_n(X)}(\mu) \cnd_\nu(\d\mu) \pp'(\d\nu)\\
			&= \sum_{n\geq 0} \iint \frac{1_{\Mpm_n(X)}(\nu)}{\thp_{\nu}\bigl( \nu \bigr)} \phi(\nu) 1_{\Mpm_n(X)}(\mu) \thp_\mu(\d\nu) \pp(\d\mu)\\
			&= \sum_{n\geq 0} \iint \frac{1_{\Mpm_n(X)}(\mu)}{\thp_{\nu}\bigl( \nu \bigr)} \phi(\mu) \thp_\mu\bigl( \nu \bigr) \pp(\d\mu)
			= \int \phi(\mu) \pp(\d\mu),
		\end{aligned}
\end{multline*}
which is the desired integration-by-parts formula.
\end{proof}

\begin{remark}	\label{rem:fpp:ibpf}
If $\thp$ is a thinning of type 2, then in Lemma~\ref{thm:ibpf:thn},
\begin{align*}
	\thp_\mu \bigl( \phi 1_{\zeta_X=n-1} \bigr) = \int \frac{Z_\mu}{Z_{\mu-\delta_x}} \cdot \frac{t_{|\mu|,|\mu|-1}}{t_{|\mu|-1,|\mu|-1}} \thp_{\mu-\delta_x}\bigl( \phi 1_{\zeta_X=|\mu|-1} \bigr) \mu(\d x),
\end{align*}
and consequently, the integration-by-parts formula adapts to
\begin{align}	\label{eq:fpp:rem-ibpf:ibpf}
	\iint g(x,\mu) \frac{Z_\mu t_{|\mu|,|\mu|-1}}{Z_{\mu-\delta_x}t_{|\mu|-1,|\mu|-1}}\mu(\d x) \pp(\d\mu) 
		= \iint g(x,\mu+\delta_x) \pi(\mu,\d x) \pp(\d\mu),
\end{align}
where now the kernel is given by
\begin{align*}
	\pi(\mu,\d x) = \frac{Z_{\mu+\delta_x} t_{|\mu|+1,|\mu|}}{Z_\mu t_{|\mu|,|\mu|} \pp^!_\mu(\zeta_X=0)} \rho(\mu,\d x).
\end{align*}
\end{remark}

%%%%%%%%%%%%%
\subsection{General thinning characterization}
%%%%%%%%%%%%%

Given a suitable pair of a finite point process and a thinning kernel, a condensing kernel $\cnd$ is determined and a partial integration formula follows. The next step is to start with a pair of thinning and condensing kernel, and to give conditions when they admit a point process which satisfies the thinning equation. For random variables, essentially the alternating cycle condition is a sufficient condition to ensure the existence of a (possibly infinite) measure given a thinning and an condensing. However, in general only the second alternating cycle condition carries over directly to the point process case due to its discrete nature. The first one only if the condensing kernel is discrete itself. For convenience, denote by $\zeta$ the identity mapping on $\Mpf(X)$.

\begin{proposition}[Alternating cycle condition]	\label{thm:fpp:alt-cycle}
Let $\pp$ be a point process, $\thp$ a thinning kernel and $\pp'$ be the thinning of $\pp$. Let $g:\Mpf\times\Mpf\times\Mpf\times\Mpf\to\R_+$ be measurable, then
\begin{multline*}
	\int g(\kappa,\mu,\lambda,\nu) \thp_\nu(\{\kappa\}) \cnd_\lambda(\d\nu) \thp_\mu(\d\lambda) \cnd_\kappa(\d\mu) \pp'(\d\kappa)\\
		= \int g(\kappa,\mu,\lambda,\nu) \thp_\mu(\{\kappa\}) \cnd_\lambda(\d\mu) \thp_\nu(\d\lambda) \cnd_\kappa(\d\nu) \pp'(\d\kappa),
\end{multline*}
i.e. for $\pp'$-a.e. $\kappa$,
\begin{align*}
	\thp_\nu(\{\kappa\}) \cnd_\lambda(\d\nu) \thp_\mu(\d\lambda) \cnd_\kappa(\d\mu)
		= \thp_\mu(\{\kappa\}) \cnd_\lambda(\d\mu) \thp_\nu(\d\lambda) \cnd_\kappa(\d\nu).
\end{align*}
\end{proposition}

\begin{proof}
Let $g$ be a non-negative, measurable function. Successive application of the disintegration~\eqref{eq:fpp:cond-disint} and change of the order of integration yields
\begin{multline*}
	\int g(\kappa,\mu,\lambda,\nu) \thp_\nu(\{\kappa\}) \cnd_\lambda(\d\nu) \thp_\mu(\d\lambda)\cnd_\kappa(\d\mu) \pp'(\d\kappa)\\
		\begin{aligned}
			&= \int g(\kappa,\mu,\lambda,\nu) \thp_\nu(\{\kappa\}) \cnd_\lambda(\d\nu) \thp_\mu(\d\lambda)\thp_\mu(\d\kappa) \pp(\d\mu)\\
			&= \int g(\kappa,\mu,\lambda,\nu) \thp_\nu(\{\kappa\}) \cnd_\lambda(\d\nu) \thp_\mu(\d\kappa) \cnd_\lambda(\d\mu) \pp'(\d\lambda)\\
			&= \int g(\kappa,\mu,\lambda,\nu) \thp_\nu(\{\kappa\}) \thp_\mu(\d\kappa) \cnd_\lambda(\d\mu) \thp_\nu(\d\lambda) \pp(\d\nu)\\
			&= \int g(\kappa,\mu,\lambda,\nu) 1_{\rho=\kappa} \thp_\nu(\d\rho) \thp_\mu(\d\kappa) \cnd_\lambda(\d\mu) \thp_\nu(\d\lambda) \pp(\d\nu)\\
			&= \int g(\kappa,\mu,\lambda,\nu) 1_{\kappa=\rho} \thp_\mu(\d\kappa) \cnd_\lambda(\d\mu) \thp_\nu(\d\lambda) \cnd_\rho(\d\nu) \pp'(\d\rho),
		\end{aligned}
\end{multline*}
and renaming $\kappa$ and $\rho$ yields the first claim. The second is an immediate consequence.
\end{proof}

The strategy to obtain the unknown point process $\pp$ from $\thp$ and $\cnd$ agrees essentially with the one for random variables in that at the process is constructed on inductively on $\Mpm_n(X)$ for $n=0,1,2,\ldots$. However, instead of constructing $\pp$, its thinning will be given in order to employ the alternating cycle condition in Proposition~\ref{thm:fpp:alt-cycle}. For $n\in\N_0$, denote by $\pp_{n}'$ a measure on $\Mpf(X)$, which is concentrated on $\Mpm_n(X)$. Finally, each of the $\pp_n'$ will be considered as the restriction of some measure $\pp'$ to $\Mpm_n(X)$. Set
\begin{gather}
	\pp_0' (\Mpm_0) \defeq 1	\label{eq:fpp:rec-0}\\
	\begin{multlined}	\label{eq:fpp:rec-n}
		\iint g(\mu) 1_{\Mpm_n(X)}(\nu) \thp_\mu(\d\nu) \cnd_\mu(\{\mu\}) \pp_{n+1}'(\d\mu) \qquad \\
		\defeq \iint g(\mu) \thp_\mu(\{\mu\}) 1_{\Mpm_{n+1}(X)}(\mu) \cnd_\nu(\d\mu) \pp_{n}'(\d\nu).
	\end{multlined}
\end{gather}
Hence $\pp_{n+1}'$ is defined with the additional density $\thp_\mu\bigl(\Mpm_n(X)\bigr)\cnd_\mu(\{\mu\})$, compare with the Recursion~\eqref{eq:gp:thm-inv-rec}. If the sum of the measures $\pp_0',\pp_1',\ldots$ yields a finite measure, then the desired (thinned) point process will be its normalization.

The first step in the context of random variables was to extend for every $n$ Equation~\eqref{eq:fpp:rec-n} to any $m<n$ with the aid of the alternating cycle condition.
\begin{lemma}	\label{thm:fpp:stepstep}
Assume that $\thp$ is a positive thinning kernel, that $\cnd$ is a condensation kernel such that $\cnd_\eta(\{\eta\})>0$ for all $\eta\in\Mpf(X)$, and let $(\pp_n)_{n\in\N_0}$ be a the family of measures given in~\eqref{eq:fpp:rec-0} and~\eqref{eq:fpp:rec-n}. Then for each $m<n$ and non-negative, measurable function $g$,
\begin{multline}	\label{eq:fpp:stepstep}
	\iint g(\mu,\nu) 1_{\Mpm_m}(\nu) \thp_\mu(\d\nu) \cnd_\mu(\{\mu\}) \pp_{n}'(\d\mu)\\
		=	\iint g(\mu,\nu) \thp_\mu(\{\mu\}) 1_{\Mpm_n}(\mu) \cnd_\nu(\d\mu) \pp_m'(\d\nu).
\end{multline}
\end{lemma}

\begin{proof}
For $m=n-1$, Equation~\eqref{eq:fpp:stepstep} is exactly the recursion, and the claim is proved by induction over $m$ for fixed $n$. Suppose that~\eqref{eq:fpp:stepstep} holds for some index $m$. Then for any non-negative function $g$, by first applying~\eqref{eq:fpp:stepstep} for $m$ and $n$ and exchanging intagrals,
\begin{multline*}
	\int g(\mu,\nu) 1_{\kappa=\nu} \thp_\eta(\d\kappa) 1_{\Mpm_{m-1}}(\nu) \thp_\mu(\d\nu)  \cnd_\eta(\{\eta\}) 1_{\Mpm_{m}}(\eta) \thp_\mu(\d\eta) \cnd_\mu(\{\mu\}) \pp_{n}'(\d\mu)\\
	\begin{aligned}
		&= \int g(\mu,\nu) 1_{\kappa=\nu} \thp_\eta(\d\kappa) 1_{\Mpm_{m-1}}(\nu) \thp_\mu(\d\nu) \cnd_\eta(\{\eta\}) \thp_\mu(\{\mu\}) 1_{\Mpm_{n}}(\mu) \cnd_\eta(\d\mu) \pp_{m}'(\d\eta)\\
		&= \int g(\mu,\nu) 1_{\nu=\kappa}  \thp_\mu(\d\nu)  \thp_\mu(\{\mu\}) 1_{\Mpm_{n}}(\mu) \cnd_\eta(\d\mu) 1_{\Mpm_{m-1}}(\kappa) \thp_\eta(\d\kappa) \cnd_\eta(\{\eta\}) \pp_{m}'(\d\eta),\\[0.5Em]
\rlap{second~\eqref{eq:fpp:stepstep} for $m$ and $m-1$,}\\[0.5Em]
		&= \int g(\mu,\nu) \thp_\mu(\{\mu\}u) 1_{\nu=\kappa}  \thp_\mu(\d\nu) 1_{\Mpm_{n}}(\mu) \cnd_\eta(\d\mu) \thp_\eta(\{\eta\}) 1_{\Mpm_{m}}(\eta) \cnd_\kappa(\d\eta) \pp_{m-1}'(\d\kappa),\\[0.5Em]
\rlap{and finally the alternating cycle condition and reordering yields}\\[0.5Em]
		&= \int g(\mu,\nu) \thp_\mu(\{\mu\}) 1_{\nu=\kappa}  \thp_\eta(\d\nu)  \cnd_\eta(\{\eta\}) 1_{\Mpm_{m}}(\eta) \thp_\mu(\d\eta) 1_{\Mpm_{n}}(\mu) \cnd_\kappa(\d\mu) \pp_{m-1}'(\d\kappa)\\
		&= \int g(\mu,\nu) 1_{\kappa=\nu}  \thp_\eta(\d\kappa)  \cnd_\eta(\{\eta\}) 1_{\Mpm_{m}}(\eta) \thp_\mu(\d\eta) \thp_\mu(\{\mu\}) 1_{\Mpm_{n}}(\mu) \cnd_\nu(\d\mu) \pp_{m-1}'(\d\nu).
	\end{aligned}
\end{multline*}
Observe that the integrals with respect to $\eta$ and $\kappa$ on the very left and on the very right agree, and this yields the claim.
\end{proof}

The following argument is roughly: Because of (to be read in reverse direction)
\begin{align*}
	\pp_n'(\d\mu) \thp_\mu(\d\nu) 1_{\Mpm_k}(\nu) &= \sum_{j\leq n} \pp_j'(\d\kappa) \cnd_\kappa(\d\mu) 1_{\Mpm_n}(\mu) \thp_\mu(\d\nu) 1_{\Mpm_k}(\nu) 	\\
		&= \sum_{j\leq n} \pp_k'(\d\nu) \cnd_\nu(\d\mu) 1_{\Mpm_n}(\mu) \thp_\mu(\d\kappa) 1_{\Mpm_j}(\kappa), 
\end{align*}
summation over $j$ removes $\thp$ on the last line, and this finally allows the disintegration. But prove first that the replacement is allowed.

\begin{lemma}	\label{thm:fpp:stepgen}
Assume that $\thp$ is a positive thinning kernel and that $\cnd$ is a condensation kernel such that $\cnd_\eta(\{\eta\})>0$ for all $\eta\in\Mpf(X)$. For all $j,m,n\in\N_0$ and non-negative, measurable functions $g$,
\begin{multline*}
	\iiint g(\kappa,\mu,\nu) 1_{\Mpm_{n}}(\mu) 1_{\Mpm_{m}}(\nu) \thp_\mu(\d\nu) \cnd_\kappa(\d\mu) \pp_j'(\d\kappa) \\
		=  \iiint g(\kappa,\mu,\nu) 1_{\Mpm_{n}}(\mu) 1_{\Mpm_{j}}(\kappa) \thp_\mu(\d\kappa) \cnd_\nu(\d\mu) \pp_m'(\d\nu).
\end{multline*}
\end{lemma}

\begin{proof}
Let $\tilde g$ be any non-negative, measurable function and assume $m\leq j$ first. Then by plugging in 
\begin{align*}
	\tilde g(\kappa,\mu,\nu) = g(\kappa,\mu,\nu) \thp_\kappa(\{\nu\}) \cnd_\kappa(\{\kappa\}),
\end{align*} 
changing the order of integration and by Lemma~\ref{thm:fpp:stepstep},
\begin{align*}
	\int g(\kappa,&\mu,\nu) 1_{\eta=\nu} \thp_\kappa(\d\eta) \cnd_\kappa(\{\kappa\}) 1_{\Mpm_{m}}(\nu) \thp_\mu(\d\nu) 1_{\Mpm_{n}}(\mu) \cnd_\kappa(\d\mu) \pp_j'(\d\kappa) \\
		&= \int g(\kappa,\mu,\nu) 1_{\nu=\eta} \thp_\mu(\d\nu) 1_{\Mpm_{n}}(\mu) \cnd_\kappa(\d\mu) \thp_\kappa(\{\kappa\}) 1_{\Mpm_{j}}(\kappa) \cnd_\eta(\d\kappa) \pp_m'(\d\eta),		\\
\intertext{which thanks to the alternating cycle condition turns into}
		&= \int g(\kappa,\mu,\nu) 1_{\nu=\eta} \thp_\kappa(\d\nu) \cnd_\kappa(\{\kappa\}) 1_{\Mpm_{j}}(\kappa) \thp_\mu(\d\kappa) 1_{\Mpm_{n}}(\mu) \cnd_\eta(\d\mu) \pp_m'(\d\eta)\\
		&= \int g(\kappa,\mu,\nu) 1_{\eta=\nu} \thp_\kappa(\d\eta) \cnd_\kappa(\{\kappa\}) 1_{\Mpm_{j}}(\kappa) \thp_\mu(\d\kappa) 1_{\Mpm_{n}}(\mu) \cnd_\nu(\d\mu) \pp_m'(\d\nu).
\end{align*}
Observe that using this computation in reverse order proves the claim for $m>j$ (with the roles of $m$ and $j$ interchanged).
%\begin{multline*}
%	\int g(\kappa,\mu,\nu) T_\eta(\zeta=\eta) 1_{\eta=\nu} Q_\kappa(\d\eta) 1_{\Mpm_{m}}(\nu) T_\mu(\d\nu) 1_{\Mpm_{n}}(\mu) Q_\kappa(\d\mu) P_j'(\d\kappa) \\
%	\begin{aligned}
%		&= \int g(\kappa,\mu,\nu) 1_{\eta=\nu} 1_{\Mpm_{m}}(\nu) T_\mu(\d\nu) 1_{\Mpm_{n}}(\mu) Q_\kappa(\d\mu) 1_{\Mpm_{j}}(\kappa) T_\eta(\d\kappa) Q_\eta(\zeta=\eta) P_m'(\d\eta) \\
%		&= \int g(\kappa,\mu,\nu) 1_{\eta=\nu} T_\nu(\zeta=\nu) 1_{\Mpm_{m}}(\nu) Q_\kappa(\d\nu) 1_{\Mpm_{j}}(\kappa) T_\mu(\d\kappa) 1_{\Mpm_{n}}(\mu) Q_\eta(\d\mu) P_m'(\d\eta) \\
%		&= \int g(\kappa,\mu,\nu) 1_{\eta=\nu} T_\nu(\zeta=\eta) 1_{\Mpm_{m}}(\eta) Q_\kappa(\d\eta) 1_{\Mpm_{j}}(\kappa) T_\mu(\d\kappa) 1_{\Mpm_{n}}(\mu) Q_\nu(\d\mu) P_m'(\d\nu).
%	\end{aligned}
%\end{multline*}
\end{proof}

\begin{theorem}
Let $\thp$ be a positive thinning kernel and $\cnd$ be a condensation kernel which satisfy the alternating cycle condition and such that if $n_0$ is the smallest index such that $\cnd_\eta(\Mp_{n_0})=0$ for some $\eta\in\Mp_n$ for some $n<n_0$, then also $\cnd_\eta(\Mp_m)=0$ for all $\eta\in\Mp_k$, $k<n_0$ and $m\geq n_0$. Then there exist $\sigma$-finite measures $\pp'$ and $\pp$ on $\Mpm(X)$ such that
\begin{align*}
	\iint g(\mu,\nu) \thp_\mu(\d\nu) \pp(\d\mu) = \iint g(\mu,\nu) \cnd_\nu(\d\mu) \pp'(\d\nu)
\end{align*}
for all non-negative, measurable functions $g$.
\end{theorem}

\begin{proof}
From Lemma~\ref{thm:fpp:stepgen} follows immediately 
\begin{multline*}
	\iint  g(\mu,\nu) 1_{\Mpm_k}(\nu) \thp_\mu(\d\nu) \pp_n(\d\mu)  \\
		\begin{aligned}
			&= \sum_{j\leq n} \iiint  g(\mu,\nu) 1_{\Mpm_k}(\nu) \thp_\mu(\d\nu) 1_{\Mpm_n}(\mu) \cnd_\kappa(\d\mu) \pp_j'(\d\kappa)  	\qquad \\
			&= \sum_{j\leq n} \iiint  g(\mu,\nu) 1_{\Mpm_j}(\kappa) \thp_\mu(\d\kappa) 1_{\Mpm_n}(\mu) \cnd_\nu(\d\mu) \pp_k'(\d\nu)  	\\
			&= \iint  g(\mu,\nu) 1_{\Mpm_n}(\mu) \cnd_\nu(\d\mu) \pp_k'(\d\nu).
		\end{aligned}
\end{multline*}
Summation over $n$ and $k$ finally yields
\begin{align*}
	\iint  g(\mu,\nu) \thp_\mu(\d\nu) \pp(\d\mu) = \iint  g(\mu,\nu) \cnd_\nu(\d\mu) \pp'(\d\nu).
\end{align*}
\end{proof}

For the finiteness of the measure $\pp$, the composed kernel $A$ of the thinning and the condensation kernel needs to be studied, since for every non-negative, measurable function $\phi$,
\begin{align*}
	\int \phi(\mu) \pp(\d\mu) = \iint \phi(\mu) \cnd_\nu(\d\mu)\pp'(\d\nu)
		= \iiint \phi(\mu) \cnd_\nu(\d\mu)\thp_\eta(\d\nu)\pp(\d\eta).
\end{align*}
Thus, a recurrence condition for $A_\eta\defeq \cnd_\nu(\scdot) \thp_\eta(\d\nu)$ needs to be checked.

%%%%%%%%%%%%%
\subsection{Examples}
\label{sect:fpp:ex}
%%%%%%%%%%%%%

First comments are devoted to the independent $q$-thinning. Note that the left hand side of Equation~\eqref{eq:fpp:thm-ibpf:ibpf} is essentially the Campbell measure of $\pp$.

\begin{corollary}
Let $\pp$ be a finite point process, $q\in(0,1)$ and $\thp$ be the independent $q$-thinning. Then for all non-negative, measurable $g$,
\begin{align*}
	\iint g(x,\mu) \mu(\d x) \pp(\d\mu)
		= \iint g(x,\mu+\delta_x) \pi(\mu,\d x) \pp(\d\mu),
\end{align*}
where the kernel $\pi$ is the kernel from $\Mpf(X)$ to $X$ induced by
\begin{align*}
	\frac{ 1_{|\eta|=1} }{ \pp^!_\mu(\zeta_X=0) } \pp^!_\mu(\d\eta).
\end{align*}
\end{corollary}

\begin{proof}
From $Z_\mu=(1-q)^{|\mu|}$ follows
\begin{align*}
	\iint g(x,\mu) \frac{Z_\mu}{Z_{\mu-\delta_x}}\mu(\d x) \pp(\d\mu)
		= (1-q) \iint g(x,\mu) \mu(\d x) \pp(\d\mu),
\end{align*}
and therefore the factor $1-q$ may be removed.
\end{proof}

The primer examples corresponding to Poisson, binomial and negative binomial distribution given in Example~\ref{ex:gp_panjer} follow. Recall that
\begin{align*}
	\thp_\mu(\phi) = (1-q)^{\mu(X)} \sum_{n\geq 0} \frac{1}{n!} \left(\frac{q}{1-q}\right)^n \int \phi(\delta_{x_1}+\ldots+\delta_{x_n}) \mu^{-[n]}(\d x_1,\ldots,\d x_n),
\end{align*}
see~\cite[Prop. 6.3.5]{bN13a}.
\begin{example}
%The independent $q$-thinning is given by the multinomial sampling
%\begin{align*}
%	T_\mu(\eta) = \prod_{j=1}^l\binom{\mu(x_j)}{\eta(x_j)} q^{\eta(x_j)}(1-q)^{\mu(x_j)-\eta(x_j)}\cdot 1_{\eta\leq\mu},
%\end{align*}
%where $\supp\mu=\{x_1,\ldots,x_l\}$. Note that
%\begin{align*}
%	t(x_1,\ldots,x_n) = q^n(1-q)^{-n},
%\end{align*}
%see~\cite[approx. Prop. 6.3.5]{bN13a}, such that
%\begin{align*}
%	T_\mu(\phi) = (1-q)^{\mu(X)} \sum_{n\geq 0} \frac{1}{n!} \int \phi(\delta_{x_1}+\ldots+\delta_{x_n}) \frac{q^n}{(1-q)^n} \mu^{-[n]}(\d x_1,\ldots,\d x_n).
%\end{align*}
\begin{enumerate}
	\item Let $\Poi_\lambda$ be a Poisson process with finite intensity measure $\lambda$, then the independent $q$-thinning $\Poi_\lambda'=\Poi_{q\lambda}$ is a Poisson process with intensity measure $q\lambda$, and since the reduced Palm kernels of a Poisson process agree with the Poisson process itself, $\cnd_{\eta}=\Poi_{(1-q)\lambda}\ast\Delta_\eta$ is a Poisson process with intensity measure $(1-q)\lambda$ with the point configuration $\eta$ added, or equivalently $\Upsilon_\eta=\Poi_{(1-q)\lambda}$.
	\item For the P\'olya sum process, $\pp=\Poy_{z,\lambda}$, where $z\in(0,1)$ and $\lambda$ is a finite measure, the independent $q$-thinning is again a P\'olya sum process $P'=\Poy_{\frac{zq}{1-qz},\lambda}$. Its reduced Palm kernels are P\'olya sum processes as well, more precisely
		\begin{align*}
			\pp_\eta^! = \Poy_{z,\lambda+\eta} = \Poy_{z,\lambda} \ast \Poy_{z,\eta}.
		\end{align*}	
		The latter equality is an independent superposition property of the P\'olya sum process. Therefore,  $\cnd_{\eta}=\Poy_{(1-q)z,\lambda+\eta}\ast\Delta_\eta$ as well as $\Upsilon_\eta=\Poy_{(1-q)z,\lambda+\eta}$.
	\item Let $z>0$ and $\lambda$ be a finite point measure. For the P\'olya difference process $\pp=\PoyD_{z,\lambda}$, the $q$-thinning is $\pp'=\PoyD_{\frac{zq}{1+qz},\lambda}$. Moreover, $\cnd_{\eta}=\PoyD_{(1-q)z,\lambda-\eta}\ast\Delta_\eta$, i.e. $\Upsilon_\eta=\PoyD_{(1-q)z,\lambda-\eta}$.
	\item Let $\pp$ be a log-Gauss Cox process $\lgcp(\mu,c)$, i.e. a Cox process whose random intensity measure is given by the exponential of a Gaussian process with mean function $\mu$ and covariance function $c$. If $\exp\left[\int \mu_x+\nicefrac{c_{x,x}}{2}\d x\right]<\infty$, then $\pp$ is a finite point process. Its independent $q$-thinning $\pp'$ is again a log-Gauss Cox process with parameters $\mu+\ln q$ and $c$. The Palm kernels were given in~\cite{CMW15a}, they are log-Gauss Cox processes again, $\pp_\eta^!$ is $\lgcp\left(\mu+\int c_{x,\scdot}\eta(\d x),c\right)$, and therefore
		\begin{align*}
			\Upsilon_\eta(\phi) = \frac{1}{\pp_\eta^!\bigl( (1-q)^{\zeta_X} \bigr) } \int \phi(\nu) (1-q)^{|\nu|}\pp_\eta^!(\d\nu).
		\end{align*}
		Further evaluation yields that the splitting kernel is Cox itself, but its directing measure is only absolutely continuous with respect to a log-Gaussian measure with its density penalizing realizations with a large mass.
\end{enumerate}
\end{example}

An inhomogeneous generalization, which employs a measurable function $q:X\to(0,1)$ instead of a constant, is straight forward. Here,
\begin{align*}
	\thp_\mu(\phi) = \e^{ \int \ln\bigl(1-q(x)\bigr) \mu(\d x) } \sum_{n\geq 0} \frac{1}{n!} \int \phi(\delta_{\vx}) \prod_{j=1}^n\left(\frac{q(x_j)}{1-q(x_j)}\right) \mu^{-[n]}(\d\vx).
\end{align*}
Note that the integration-by-parts formula is
\begin{align*}
	\iint g(x,\mu) (1-q_x)\mu(\d x) \pp(\d\mu)
		= \iint g(x,\mu+\delta_x) \pi(\mu,\d x) \pp(\d\mu)
\end{align*}
with $\pi$ induced by
\begin{align*}
	\bigl( 1-\eta(q) \bigr) \frac{ 1_{|\eta|=1} }{ \pp^!_\mu(\zeta_X=0) } \pp^!_\mu(\d\eta).
\end{align*}
The given examples carry over directly.

Thinnings of type 2 are in spirit close to mixed sample processes. Recall that these thinnings,
\begin{align*}
	\thp_\mu(\phi) = Z_\mu \sum_{m\geq 0} \frac{t_{|\mu|,m}}{m!} \int \phi(\delta_\vy) \mu^{-[m]}(\d\vy),
\end{align*}
have a doubly stochastic nature: Firstly, some number of points independent of their positions is chosen to survive, and secondly a subset of these points uniformly among all subsets of that given size is chosen. Then
\begin{align*}
	Z_\mu^{-1} = \sum_{j=0}^{|\mu|} \binom{|\mu|}{j} t_{|\mu|,j},\qquad
	\frac{Z_\mu}{Z_{\mu-\delta_x}} = \frac{ \sum_{j=0}^{|\mu|-1} \binom{|\mu|-1}{j} t_{|\mu|-1,j} }{ \sum_{j=0}^{|\mu|} \binom{|\mu|}{j} t_{|\mu|,j} }.
\end{align*}
Since $Z_\mu$ depends on $\mu$ only via its total mass, write also $Z_m$ instead of $Z_\mu$ for $\mu\in\Mpm_m$. 

The Palm kernel of mixed sample processes are mixed sample themselves, and thus the Papangelou kernel is almost as simple as the one for the Poisson process.

\begin{corollary}
Let $\thp$ be a thinning of type 2 and $\pp$ be a mixed sample process such that
\begin{align*}
	\pp(\phi) = \frac{1}{\Xi} \sum_{n\geq 0} \frac{p_n}{n!} \int \phi(\delta_\vx) \lambda^n(\d\vx)
\end{align*}
for a sequence $(p_n)_{n\in\N_0}$ of non-negative numbers such that if $p_{n_0}=0$, then $p_n=0$ for all $n\geq n_0$, and $\lambda$ is a probability measure on $X$. Then $\pp$ satisfies an integration-by-parts formula
\begin{align*}
	\iint g(x,\mu) \delta(\mu,x)\mu(\d x) \pp(\d\mu) 
		= \iint g(x,\mu+\delta_x) \pi(\mu,\d x) \pp(\d\mu)
\end{align*}
for non-negative, measurable functions $g$, where for $\pp$-a.e. $\mu\in\Mpm_m(X)$,
\begin{align*}
	\pi(\mu,\d x) = \frac{Z_{m+1}}{Z_m} \frac{t_{m+1,m}}{t_{m,m}} \frac{p_{m+1}}{p_m} \lambda(\d x), \qquad
	\delta(\mu,x) = \frac{Z_m}{Z_{m-1}}\cdot \frac{t_{m,m-1}}{t_{m-1,m-1}}.
\end{align*}
\end{corollary}

\begin{proof}
Note that the $m$-th factorial moment measure $\nu_\pp^m$ of $\pp$ is given by
\begin{align*}
	\nu_\pp^m (f) = \frac{1}{\Xi} \sum_{k\geq 0} \frac{p_{m+k}}{k!} \int f(\vy)\lambda^m(\d\vy) \eqdef \iota_m \lambda^m(f).
\end{align*}
Thus, the disintegration of the $m$-th reduced Campbell measure of $\pp$ with respect to $\nu_\pp^m$ yields
\begin{align*}
	C^{!,m}_\pp(g) &= \frac{1}{\Xi} \sum_{k\geq 0} \frac{p_{m+k}}{k!} \iint g(\vy,\delta_\vx) \lambda^k(\d\vx)\lambda^m(\d\vy),
\end{align*}
and the reduced Palm kernel is for the point configuration $\delta_\vy$ consisting of $m$ points
\begin{align*}
	\pp^!_{\vy}(\phi) = \frac{1}{\Xi_m} \sum_{k\geq 0} \frac{p_{m+k}}{k!} \int \phi(\delta_\vx) \lambda^k(\d\vx)
\end{align*}
and $\Xi_m$ is the corresponding normalization constant. Thus, plugging this into Remark~\ref{rem:fpp:ibpf} yields the Papangelou kernel $\pi$.
\end{proof}

\begin{example}
Let $\pp$ be a mixed sample process with weight sequence $(p_n)_{n\in\N_0}$ and a probability measure $\lambda$ on $X$, and $\thp$ be a thinning of type 2. A short calculation shows that $\pp'$ is mixed sample again, with weights
\begin{align*}
	p_m'= \frac{1}{\Xi} \sum_{k\geq 0} \frac{p_{m+k}Z_{m+k}t_{m+k,m}}{k!},
\end{align*}
i.e.
\begin{align*}
	\thp\pp(\phi) &= \frac{1}{\Xi} \sum_{m\geq 0} \frac{1}{m!} \sum_{k\geq 0} \frac{p_{m+k}Z_{m+k}t_{m+k,m}}{k!} \int \phi(\delta_{\vy}) \lambda^m(\d\vy).
\end{align*}
For simplicity, assume subsequently that $t_{n,j} = \binom{n}{j}^{-1}$, such that $Z_\mu=\nicefrac{1}{|\mu|+1}$ and $\nicefrac{Z_\mu}{Z_{\mu-\delta_x}} = \nicefrac{ |\mu| }{ |\mu|+1 }$. Then the integration-by-parts formula is
\begin{align*}
	\iint g(x,\mu) \frac{1}{|\mu|+1} \mu(\d x) \pp(\d\mu)
		= \iint g(x,\mu+\delta_x) \frac{p_{|\mu|+1}}{\bigl(|\mu|+2\bigr)p_{|\mu|}} \lambda(\d x) \pp(\d\mu).
\end{align*}
Comparing to Example~\ref{ex:gp_power}, the uniform thinning together with power laws, if $p_k=k!\cdot (k+1)^{-\alpha}$, the kernel on the right hand side is $\pi(\mu,\d x) = \left(\nicefrac{m+1}{m+2}\right)^{\alpha+1}\lambda(\d x)$ whenever $\mu\in\Mpm_m$. If the function $(x,\mu)\mapsto g(x,\mu)$ depends on the total mass of $\mu$ only, the integration-by-parts formula in Example~\ref{ex:gp_power} is recovered.
\end{example}

Finally, let $V:X\times\Mpf(X)\to\R_+$ be measurable and (abusing notation) denote also by $V$ the function
\begin{align}	\label{eq:fpp:ex:interaction}
	V(\vx) = V(x_1;0) + V(x_2;\delta_{x_1}) + \ldots + V(x_n;\delta_{x_1} + \ldots + \delta_{x_{n-1}})).
\end{align}
Assuming that the following line defines a probability measure,
\begin{align}	\label{eq:fpp:ex:interact-pp}
	\pp(\phi) = \frac{1}{\Xi} \sum_{n\geq 0} \frac{1}{n!} \int \phi(\delta_\vx) \e^{-V(\vx)} \lambda^n(\d\vx),
\end{align}
$\pp$ is a point process with interaction. Note that
\begin{align*}
	\pp^!_\nu(\phi) = \frac{1}{\Xi_\nu} \sum_{n\geq 0} \frac{1}{n!} \int \phi(\delta_\vx) \e^{-V_\nu(\vx)}\lambda^n(\d\vx),
\end{align*}
where $V_\nu$ is given similar to $V$ in~\eqref{eq:fpp:ex:interaction} with initialized with the point configuration $\nu$ instead of the empty configuration and $\Xi_\nu$ is the corresponding normalization constant.

\begin{example}
Let $\thp$ be a thinning of type 1 with weight function $t(\vy)=\e^{-\tilde V(\vy)}$ for a function $\tilde V$ similar to $V$ given in~\eqref{eq:fpp:ex:interaction}, and $\pp$ be a point process given in~\eqref{eq:fpp:ex:interact-pp}. Then
\begin{align*}
	\thp\pp(\phi) = \frac{1}{\Xi} \sum_{m\geq 0} \frac{1}{m!} \int \phi(\delta_\vy) r(\vy) \e^{-\tilde V(\vy)-V(\vy)} \lambda^m(\d\vy),
\end{align*}
where $r(\vy) = \Xi_{\delta_\vy}\pp^!_{\delta_\vy}(Z_{\delta_{\vy+\scdot}})$. Moreover, since $\pp^!_\nu(\zeta_X=0)=\nicefrac{1}{\Xi_\nu}$ for $\pp$-a.e. $\nu\in\Mpf$, the splitting kernel is
\begin{align*}
	\Upsilon_\nu(\phi) = \frac{1}{\pp^!_\nu(\zeta_X=0)} \int \phi(\eta) Z_{\nu+\eta} \pp^!_\nu(\d\eta)
		= \frac{1}{\gamma_\nu}\sum_{n\geq 0} \frac{1}{n!} \int \phi(\delta_\vy) \e^{-V_\nu(\vy)} Z_{\delta_\vy+\nu} \lambda^n(\d\vy),
\end{align*}
thus $\pi$ turns out to be
\begin{align*}
	\pi(\mu,\d x) = \frac{Z_{\mu+\delta_x}}{Z_\mu} \e^{-V_\mu(x)} \lambda(\d x).
\end{align*}
\end{example}

Note that in general not all pairs of point processes and thinnings admit an integration-by-parts formula, e.g. if $\pp$ is a finite Poisson process and $\thp$ is a thinning which deletes all points whose next neighbour is at some distance smaller than a fixed $\delta>0$. In this case, removing a single point from a realization gets weight 0 with a positive probability.

%%%%%%%%%%%%%%%%%
% BACK MATTER
%%%%%%%%%%%%%%%%%

\end{document}